\documentclass[reqno,11pt, letterpaper]{amsart}
\PassOptionsToPackage{fleqn}{amsmath}
\RequirePackage{amsthm}
\RequirePackage{amsmath}
\RequirePackage{latexsym}
\RequirePackage{amssymb}
\RequirePackage{amscd}
\RequirePackage{epsfig}
\RequirePackage{graphics}
\RequirePackage{ifthen}
\RequirePackage{varioref}
\usepackage{changes}
\usepackage{soul}
\usepackage{xcolor}
\usepackage{color}
\usepackage[misc]{ifsym}
\usepackage{bm}
\usepackage{lipsum}
\usepackage{subfigure}
\usepackage{enumitem}

\usepackage{epstopdf}
\definecolor{darkgreen}{rgb}{0.0625,0.64,0.0625}
\ifpdf
  \usepackage[
    pdftex,
    colorlinks,%
    linkcolor=blue,citecolor=blue,urlcolor=blue,
    hyperindex,%
    plainpages=false,%
    bookmarksopen,%
    bookmarksnumbered%
  ]{hyperref}
\else
  \usepackage{hyperref}
\fi
\usepackage{bookmark}
\allowdisplaybreaks[1]

\def\R{{\mathbb R}}

\def\E{{\mathbb E}}

\theoremstyle{plain}
\newtheorem{thm}{Theorem}[section]
\newtheorem{lem}[thm]{Lemma}
\newtheorem{prop}[thm]{Proposition}
\newtheorem{cor}[thm]{Corollary}
\theoremstyle{definition}
\newtheorem{rem}[thm]{Remark}

\newtheorem{defn}[thm]{Definition}

\def\d{\mathrm d}

\def\Bbb#1{{\mathbb#1}}

\def\R{\Bbb R}

\def\S{\Bbb S}

\numberwithin{equation}{section}
\newcommand\blfootnote[1]{%
 \begingroup
 \renewcommand\thefootnote{}\footnote{#1}%
 \addtocounter{footnote}{-1}%
 \endgroup
 }

\begin{document}

\title[Evolution of hypersurfaces in $(n+1)$-dimensional light-cone]{Evolution of hypersurfaces in $(n+1)$-dimensional light-cone}

	\author[X. Jiang]{Xinjie Jiang}% ${}^*$}%\footnote{Corresponding author: }
\address{Xinjie Jiang}
\email{serge0912@icloud.com}

\author[S. Pan]{Shengliang Pan}% ${}^*$}%\footnote{Corresponding author: }
\address{Shengliang Pan}
\email{slpan@tongji.edu.cn}

\address{
School of Mathematical Sciences, Key Laboratory of Intelligent Computing and Applications (Ministry of Education), Tongji University, No.1239, Siping Road, Shanghai, 200092, P. R. China}

    \author[Y. Yang]{Yun Yang${}^*$}\blfootnote{${}^*$~Corresponding author: yangyun@mail.neu.edu.cn}
    \address{ Yun Yang\newline\indent
     Department of Mathematics, Northeastern University, Shenyang, 110819, P.R. China}
    \email[Corresponding author]{yangyun@mail.neu.edu.cn}

\begin{abstract}
In this paper, we investigate the evolutionary processes of hypersurfaces within half of the $(n+1)$-dimensional light-cone. Depending on the evolutionary processes, our focus extends to exploring variational problems associated with a smooth function $f(S_1,\cdots,S_n)$, where each $S_r$ denotes the $r$-th elementary symmetric polynomial, defined as the sum of all possible products of $r$ distinct principal curvatures. We present several fundamental properties related to these variational problems. Furthermore, we examine a curvature-type flow defined locally within the light-cone, establishing its perpetual existence and smooth convergence to a circle whose length is preserved and equal to that of the initial curve.
\end{abstract}

\subjclass[2010]{53A35, 53A55, 53A10,  53C50}

\keywords{Minkowski light-cone;\; variational formula;\; constrained flow}
%Please type here List of Keywords for your article separated by semicolon.

%\Classification{53A15;\; 53A55} % e.g. 35A30; 81Q05
% For 2010 Mathematics Subject Classification see http://www.ams.org/mathscinet/msc/msc2010.html

%\linenumbers
\maketitle
\section{Introduction}
For centuries, philosophers and scientists have sought to reduce the laws of nature to a minimal set of fundamental principles. Among them, the principle most often regarded as universal is the principle of least action. In broad terms, it states that nature tends to minimize certain action variables or functionals. The mathematical framework that captures such ideas is the calculus of variations, whose central tool is the Euler–Lagrange equation, used to determine functions that maximize or minimize functionals.
This variational viewpoint, originating in physics, naturally extends to geometry. Just as the least action principle characterizes physical trajectories, the notion of minimizing functionals leads in geometry to the study of minimal submanifolds. These submanifolds arise as critical points of the area or volume functional \cite{bla0, bla1, cm}, and they also appear as static solutions to the mean curvature flow \cite{hui-84, zhu}. As a central theme in global differential geometry, minimal submanifolds have been extensively studied, producing many deep results (see \cite{bre12, bre131, bre21, bre} and references therein). A wide range of methods—such as the variational approach, Min–Max theory, geometric measure theory, and mean curvature flow—have been developed to further explore their properties \cite{mn14, mn16, mn17}.
Moreover, the study of the second variation provides crucial insights into stability, as discussed in \cite{cl0, cl}. Undoubtedly, the theory of minimal submanifolds has become a cornerstone of geometric analysis, serving as both a fundamental concept and a powerful tool in mathematics \cite{bre13}. Its well-established framework has also found significant applications in diverse fields, including computer vision, probability theory, and general relativity.

The problem of finding a minimal surface with a prescribed boundary has long drawn the attention of mathematicians such as Lagrange, Euler, and Plateau. Plateau, in particular, studied soap films stretched across wire frames, offering a striking physical realization of minimal surfaces. This classical question is now known as Plateau’s problem \cite{cour, dou, rad}. Well-known examples of minimal surfaces include the helicoid and the catenoid \cite{cm0}.
Closely related in spirit is the Bernstein problem \cite{alm, bgg}, which concerns entire minimal graphs and highlights the tension between local conditions and global geometry. While Plateau’s problem focuses on how prescribed boundary data determine the resulting surface, the Bernstein problem explores whether global regularity can be guaranteed in the absence of boundary constraints. Together, they illustrate two complementary aspects of minimal surface theory: the influence of boundary conditions and the constraints imposed by global structure. These themes have inspired extensive studies of Bernstein-type problems in Euclidean spaces (see \cite{fs, ssy, sol} and the references therein).

A natural extension of classical minimal surface theory is to study minimal surfaces in Riemannian manifolds, going beyond the Euclidean setting.
Let $M^n$ be an $n$-dimensional Riemannian manifold that is immersed isometrically into a space form $N^{n+1}$.
Denote the principal curvatures of the immersion by $\kappa_1,\cdots,\kappa_n$. For each $r=0,1,\cdots, n$, let $S_r$
be the $r$-th elementary symmetric polynomial of the principal curvatures, which is given by the sum of all possible products of $r$ distinct principal curvatures: $\displaystyle \sum_{i_1<i_2<\cdots<i_r}\kappa_{i_1}\kappa_{i_2}\cdots\kappa_{i_r}$.
Reilly \cite{rei} extended variational problems associated with area or volume functionals to integrals of the form $\displaystyle \int_Mf(S_1,\cdots,S_n)dV$, where $f$ is any smooth function defined on $M^n$, and $dV$ is the volume element of $M^n$.
The Willmore submanifold \cite{bry}  is an extremal submanifold determined by the Willmore
functional $\displaystyle \int_M (S-nH^2)^{n/2}dV$, where $S$  denotes the square of the length of the second fundamental
form, $H$ is the mean curvature of $M^n$. Notably, it is invariant under M\"obius (or conformal) transformations of $\mathbb{S}^{n+1}(1)$.

Unlike space forms equipped with a Riemannian metric, 
the Minkowski light-cone carries a degenerate metric. 
In polar coordinates, the $(n+2)$-dimensional Minkowski space can be expressed as  
\[
\mathbb{E}^{n+2}_1 \cong \mathbb{R} \times (0,\infty) \times \S^n,
\]  
with metric  
\[
\eta = -dt^2 + dr^2 + r^2 d\Omega^2,
\]  
where \(d\Omega^2\) denotes the standard round metric on \(\S^n\). 
The standard light-cone centered at the origin is given by  
\[
LC := \{|t| = r\},
\]  
and the induced metric on \(LC\) takes the degenerate form \(r^2 d\Omega^2\). 
Consequently, the usual Riemannian framework is not well-defined on the light-cone. 
Moreover, due to the normal–tangent duality, any vector orthogonal to the light-cone is necessarily tangent to it, and hence null. This feature makes the definition of the second fundamental form of the light-cone particularly subtle. In this paper, we investigate variational problems associated with a smooth function  
\[
f(S_1, \dots, S_n),
\]  
defined on hypersurfaces \(M^n\) contained in one half of the light-cone. 
In recent years, geometric flows on the light-cone have drawn significant attention 
(see \cite{rs, wol}). 
Here, we employ geometric flows as a central tool to study such variational problems.

%Let $S_r$ denote the $r$-th elementary symmetric function of the the eigenvalues
%    $k_1,\cdots,k_n$ of matrix $\Phi=\left(h^i_j\right)$, specifically:
%    \begin{align*}
%        S_0=1,\;S_1=k_1+\cdots+k_n,\;\cdots, \; S_n=k_1\cdots k_n.
%    \end{align*}
%    The Newton transformation $T_r$ are then defined inductively in the following manner
%    \begin{align*}
%      T_0{}^i_j=\delta^i_j,\; T_{r+1}{}^i_j=S_{r+1}\delta^i_j-T_{r}{}^{ik}h_{kj},\; r=0,1,\cdots, n-1.
%    \end{align*}
%
%\begin{rem}
%The hypersurface $M^n$  lies entirely in a hyperplane $P(\mathbf{v},1)$ if and only if it is an umbilical  hypersurface with constant curvature (for more details see Proposition \ref{prop-plane}). Specifically, at every point on $M^n$, all principal curvatures attain an identical value:
%\begin{align*}
%  k_1=\cdots=k_n=\frac{\eta(\mathbf{v},\mathbf{v})}{2}.
%\end{align*}
%The principal curvatures are positive when  $\mathbf{v}$ is space-like, zero when $\mathbf{v}$ is light-like, and negative when $\mathbf{v}$ is time-like.
%\end{rem}

Consider a smooth family of immersions $\mathbf{X}:M^n\times[0,T)\rightarrow LC^*$ representing the motion of a hypersurface in the future-pointing half of the light-cone $LC^*$.
The evolution of $\mathbf{X}$ is governed by
    \begin{equation*}\label{basic-flow}
       \dot{\mathbf{X}}(\theta,t)= U(\theta,t) \mathbf{X}(\theta,t)+W^i(\theta,t)\mathbf{e}_i(\theta,t), \quad \theta=\{\theta^1,\cdots,\theta^n\},
    \end{equation*}
    with the initial immersion $\mathbf{X}_0=\mathbf{X}(p,0)$, where $U(\theta,t)$ and $W^i(\theta,t)$ are scalar-valued functions describing the normal and tangential velocities, respectively, and $\displaystyle \dot{\mathbf{X}}=\frac{dr}{dt}\frac{\partial}{\partial r}+\frac{d\theta^i}{dt}\frac{\partial}{\partial\theta^i}$.
We focus on deformations that strongly fix the boundary $\partial M$, meaning that both $U$ and its gradient vanish on $\partial M$, and the tangential component of the deformation vector field $\dot{\mathbf{X}}$ is tangent to $\mathbf{X}(\partial M)$ along the boundary. Clearly, if $M$ is compact without boundary, there are no restrictions on the deformation. In this setting, the following first variation formula holds
\begin{thm}\label{main thm A}
       Suppose that $f$ is any smooth function of $n$ variable. Then
       \begin{align*}
         \frac{d}{dt}\int_Mf(S_1,\cdots,S_n)dV=\int_MU\left(\sum_{r=1}^n\left(D_rf\right)_{,ij}\left(T_{r-1}\right)^{ij}+nf-2rS_r\right)dV.
       \end{align*}
\end{thm}
Here $T_r$ denotes the Newton transformation, defined inductively by
%\begin{align*}
%	T_0=I,\quad T_{r+1}=S_{r+1}I-\Phi T_r, \quad r=0,1,\cdots, n-1.
%\end{align*}
$$
T_0{}^i_j=\delta^i_j,\quad 
T_{r+1}{}^i_j=S_{r+1}\,\delta^i_j - T_{r}{}^{ik} h_{kj}, \quad r=0,1,\cdots, n-1.
$$

\begin{defn} A hypersurface in $LC^*$ is said to be $r$-extremal if $S_r$ vanishes identically.
\end{defn}
For $f = S_r$ with $n \neq 2r$, any extremum of the functional $\int_M S_r\,dV$ is attained precisely at an $r$-extremal hypersurface.
\begin{cor}\label{1 var}
       Suppose that $M^n$ is as before. Then
       \begin{align*}
          \frac{d}{dt}\int_MS_rdV=\int_M U(n-2r)S_rdV.
       \end{align*}
       If, in addition, the immersion yields an extremal value for $n\neq 2r$, i.e., $S_r=0$ when $t=t_0$, then at $t=t_0$,
       \begin{align*}
         \frac{d^2}{dt^2}\int_MS_rdV=(2r-n)\int_M\left(T_{r-1}\right)^{ij}U_iU_jdV.
       \end{align*}
\end{cor}

\begin{rem}  When $n=1$, the $1$-extremal curves are minimal. When $n>2$, the $1$-extremal hypersurfaces are maximal. When $2r=n$, the integral $\displaystyle \int_MS_rdV$
remains invariant under previously introduced admissible deformations.
\end{rem}

The study of nonlocal constrained curve flows can be traced back to the 1980s.  In order to understand the motion of an elastically stretched rubber band enclosing incompressible fluid, Gage \cite{Gag85} incorporated a nonlocal term into the velocity of the curve shortening flow to constrain the area enclosed by the evolving curve. Specifically, he proposed the flow
\begin{align*}
	\frac{\partial\mathbf{X}}{\partial t}=\left(\frac{2\pi}{\mathcal{L}(t)}-\kappa\right)\nu,
\end{align*}
where $\kappa$ is the curvature, $\nu$ is the outward-pointing unit normal vector, and $\mathcal{L}(t)$ denotes the length of the evolving curve. Gage showed that this flow preserves the enclosed area and smoothly deforms the initial curve into a round circle.
The idea of introducing appropriate nonlocal correction terms to constrain the evolution of geometric quantities has since been widely generalized and developed. See \cite{TW15,GPT20.1,GPT20.2,GPT21} and references therein. 
%Moreover, nonlocal curvature flows with geometric constraints have been explored in other geometric settings, such as hyperbolic geometry \cite{WY22} and affine geometry \cite{YZ25}. 
In higher dimensions, the natural analogue of Gage's area-preserving flow is the volume-preserving mean curvature flow, introduced by Huisken \cite{Hui87}. He proved that this flow exists globally and evolves strictly convex initial hypersurfaces smoothly into round spheres. By suitably modifying the nonlocal term, other geometric quantities can also be constrained, see \cite{Coy04} for further developments. More general extensions have been established in \cite{And01, Coy05, CM07, Sin15, AW19}, among others.

Beyond using nonlocal terms to constrain the evolution of geometric quantities, Guan and Li \cite{GL15} employed the Minkowski identity to construct a mean curvature type flow in space forms that preserves the volume enclosed by hypersurfaces. In the Euclidean space $\mathbb{R}^{n+1}$, the flow is given by
\begin{align*}
	\frac{\partial\mathbf{X}}{\partial t}=\left(n-Hu\right)\nu,
\end{align*}
where $H$ denotes the mean curvature and $u$ is the support function. A notable feature of this flow is that it is locally constrained—the velocity is defined pointwise and does not involve any global terms. Compared to classical nonlocal constrained flows, this flow imposes weaker restrictions on the initial hypersurface, and the $C^0$-estimates come for free due to the maximum principle. Guan and Li proved that the flow exists for all time and converges smoothly to a round sphere. This result was later extended to warped product spaces in joint work with Wang \cite{GLW19}. Further generalizations can be found in \cite{WX21}. Additionally, Brendle, Guan, and Li introduced an inverse-type locally constrained curvature flow, see related work in \cite{SX19, HLW22}. In the second part of this paper, we propose a locally constrained curvature type curve flow defined on the light-cone
\begin{align}\label{length pre flow 1}
	\frac{\partial\mathbf{X}}{\partial t}=\left(r^{\frac{1}{2}}\kappa+\frac{1}{2}r^{-\frac{3}{2}}\right)\mathbf{X},
\end{align}
where $r$ is the  radial function, and $\kappa$ denotes the curvature of the curve on the light-cone. Our main result is as follows.
\begin{thm}\label{thm behavior}
	Let $\mathbf{X}_0$ be a smooth, closed, space-like curve on the future-pointing light-cone $LC^*$. Then the flow \eqref{length pre flow 1} exists for all time and converges smoothly to a circle whose length equals that of the initial curve.
\end{thm}

The paper is organized as follows. In Section \ref{sec-pre}, we derive the structure equations, the integrability conditions, and the relations among the associated quantities on the $(n+1)$--dimensional light-cone $LC^*$. In Section \ref{sec-prop}, we study several properties of hypersurfaces in $LC^*$. In Section \ref{sec-evo}, we investigate the evolution of hypersurfaces under a geometric flow on $LC^*$ and establish variational formulas for integrals of the form $\displaystyle \int_M f(S_1,\dots,S_n)dV$. Finally, in Section \ref{sec-flow}, we consider a curvature-type flow defined locally in the light-cone, and prove long-time existence and smooth convergence to a circle with length preserved equal to that of the initial curve.

\section{Preliminaries}\label{sec-pre}
Let $V$ be an $(n+2)$-dimensional real vector space. A {\it Lorentzian scalar product} on $V$ is a non-degenerate symmetric bilinear form
$\langle\cdot,\cdot\rangle_L$ of index $1$. Equivalently, there exists a basis $\mathbf{e}_1,\cdots,\mathbf{e}_{n+2}$ of $V$ such that
\begin{align*}
  \langle\mathbf{e}_i,\mathbf{e}_j\rangle_L=\left
                \{\begin{array}{cl}
                     -1, & \mathrm{if}\quad i=j=1,  \\
                     1, & \mathrm{if}\quad i=j=2,\ \cdots,\ n+2, \\
                     0, & \mathrm{otherwise}.
                   \end{array}
  \right.
\end{align*}
A {\it Lorentzian manifold} is a pair $(\mathcal{M}, \eta)$, where $\mathcal{M}$ is an $(n+2)$-dimensional smooth
manifold and $\eta$ is a Lorentzian metric, i.e., $\eta$ assigns to each point $p\in \mathcal{M}$ a
Lorentzian scalar product $\eta_p$ on the tangent space $T_p\mathcal{M}$.

The {\it Minkowski space} $\E_1^{n+2}$ is the vector space $\R^{n+2}$ endowed with the Lorentzian scalar product
\[
\langle \mathbf{u}, \mathbf{v} \rangle_L = -u_1 v_1 + u_2 v_2 + \cdots + u_{n+2} v_{n+2}, \quad 
\mathbf{u},\mathbf{v} \in \mathbb{R}^{n+2}.
\]
In Cartesian coordinates $(x_1,\dots,x_{n+2})$, the corresponding {\it Minkowski metric} reads
\[
\eta = -(dx_1)^2 + (dx_2)^2 + \cdots + (dx_{n+2})^2.
\]
A vector $\mathbf{v}=(v_1,\cdots,v_{n+2})\in\E_1^{n+2}$ is classified  according to the sign of $\langle \mathbf v,\mathbf v\rangle_L$:
%In Minkowski space $\E_1^{n+2}$, a vector $\mathbf{v}=(v_1,\cdots,v_{n+2})$ can be classified according to the sign of $\langle \mathbf v,\mathbf v\rangle_L$:
\begin{itemize}
	\item {\it space-like}, if $\langle \mathbf v,\mathbf v\rangle_L >0$ or $\mathbf v=0$;
	\item {\it time-like}, if $\langle \mathbf v,\mathbf v\rangle_L <0$;
	\item {\it light-like} (or {\it null}), if $\langle \mathbf v,\mathbf v\rangle_L=0$ and $\mathbf v\neq 0$.
\end{itemize}
Its magnitude (or length) is defined by $|\mathbf{v}| = \sqrt{|\langle \mathbf{v},\mathbf{v} \rangle_L|}$. Similarly, a \textit{hyperplane }
\[
P(\mathbf{v},c) = \{ \mathbf{x} \in \mathbb{E}_1^{n+2} \mid \langle \mathbf{x}, \mathbf{v} \rangle_L = c \}, \quad \mathbf{v}\neq 0,\, c\in\mathbb{R},
\]
is called space-like, time-like, or light-like according to the type of its \textit{pseudo-normal} vector $\mathbf{v}$.

The Minkowski space contains the following standard submanifolds $(c>0)$:
\begin{itemize}
	\item the {\it hyperbolic space}
	\[
	H^{n+1}(c)=\{\mathbf{x}\in\E^{n+2}_1 \mid \langle\mathbf{x},\mathbf{x}\rangle_L=-c^2\};
	\]
	\item the {\it de Sitter space}
	\[
	S^{n+1}_1(c)=\{\mathbf{x}\in\E^{n+2}_1 \mid \langle\mathbf{x},\mathbf{x}\rangle_L=c^2\};
	\]
	\item the {\it light-cone}
	\[
	LC=\{\mathbf{x}\in\E^{n+2}_1\setminus\{0\} \mid \langle\mathbf{x},\mathbf{x}\rangle_L=0\}.
	\]
\end{itemize}
These submanifolds are collectively known as {\it pseudo-spheres} in $\E^{n+2}_1$.

In polar coordinates, $\E^{n+2}_1$ is represented as $\R\times(0,+\infty)\times\S^n$ with metric
\begin{align*}\label{etadef}
	\eta=-dt^2+dr^2+r^2d\Omega^2,
\end{align*}
where
\begin{align*}
	d\Omega^2\;=\;(d\theta^1)^2+\sin^2\theta^1(d\theta^2)^2+\cdots+\sin^2\theta^1\cdots\sin^2\theta^{n-1}(d\theta^n)^2
\end{align*}
denotes the standard round metric on $\S^n$.
The future-pointing light-cone centered at the origin, denoted as $LC^*$, can be parameterized as
\[
\mathbf{X}=\bigl(r,\; r\cos\theta^1,\; r\sin\theta^1\cos\theta^2,\; r\sin\theta^1\sin\theta^2\cos\theta^3,\;\ldots,\; r\sin\theta^1\cdots sin\theta^n \bigr).
\]
A natural basis of its tangent space is given by  
\begin{align*}
	\Biggl\{ \frac{\partial}{\partial \theta^0} = \frac{\partial}{\partial r},\,
	\frac{\partial}{\partial \theta^1},\,\ldots,\,
	\frac{\partial}{\partial \theta^n} \Biggr\}.
\end{align*}
Define the vector field  
\begin{align*}
	\mathbf{Y} = \tfrac{1}{2}\bigl(-1,\, \cos\theta^1,\, 
	\sin\theta^1\cos\theta^2,\, \ldots,\, 
	\sin\theta^1\cdots\sin\theta^n\bigr).
\end{align*}
Then $\mathbf{Y}$ is a null vector field satisfying
\begin{align*}
	\eta\!\left(\mathbf{Y},\,\frac{\partial}{\partial r}\mathbf{X}\right) = 1, \qquad
	\eta\!\left(\mathbf{Y},\,\frac{\partial}{\partial \theta^i}\mathbf{X}\right) = 0, 
	\quad \forall\, i=1,\ldots,n.
\end{align*}

Consider a local embedding  
\[
\mathbf{X} : \S^n \;\rightarrow\; M^n \;\hookrightarrow\; LC^*.
\]
Let $D$ be the Levi-Civita connection on $\S^n$. 
For a smooth function 
$r(\theta^1,\ldots,\theta^n)$ defined on $\S^n$, we denote by $D_i r$ its covariant 
derivative in the direction $\displaystyle \frac{\partial}{\partial\theta^i}$. Then the tangent vector fields of $M^n$ induced by the embedding are
\begin{align*}
	\mathbf{e}_i \triangleq \mathbf{X}_*\left(\frac{\partial}{\partial \theta^i}\right)
	= D_i r \frac{\partial}{\partial r} + \frac{\partial}{\partial \theta^i}, \quad 1 \le i \le n,
\end{align*}
and the induced metric takes the form
\[
g = r^2 d\Omega^2,
\]
with components
\[
g_{11} = r^2, \quad g_{ii} = r^2 \prod_{l=1}^{i-1} \sin^2 \theta^l \quad (i>1), \quad g_{ij} = 0 \quad (i \neq j).
\]
The corresponding Christoffel symbols are
\begin{align*}
	&\Gamma^i_{ii} = \frac{r_i}{r}, 
	\quad \Gamma_{ij}^k = 0 \;\; (i \neq j, i \neq k, j \neq k), \\
	&\Gamma_{ii}^k = -\frac{r_k}{r} \Biggl(\prod_{l=i}^{k-1} \sin^2 \theta^l \Biggr)^{-1}, \quad i<k, \\
	&\Gamma_{ii}^k = -\frac{r_k}{r} \prod_{l=k}^{i-1} \sin^2 \theta^l - \frac{\sin(2\theta^k)}{2} \prod_{l=k+1}^{i-1} \sin^2 \theta^l, \quad i>k, \\
	&\Gamma_{ij}^i = \frac{r_j}{r} + \frac{\cos \theta^j}{\sin \theta^j}, \quad i>j, \\
	&\Gamma_{ij}^i = \frac{r_j}{r}, \quad i<j.
\end{align*}

    We now compute the null vector field $\mathbf{L}$ corresponding to $\mathbf{X}$, which, together with $\mathbf{X}$, forms a normal frame for $M^n$.
     \begin{prop}\label{prop-unique}
     	There exists a unique vector $\mathbf{L}$ satisfying that
     	\begin{align*}
     		\eta(\mathbf{L},\mathbf{X})=1,\; \eta(\mathbf{L},\mathbf{L})=0,\;\eta(\mathbf{L},\mathbf{e}_i)=0,\forall\;i=1,\cdots,n.
     	\end{align*}
     \end{prop}
     \begin{proof}
     	Assume that the vector $\mathbf{L}$ takes the form $\mathbf{L}=(a,a_2,\dots,a_{n+2})$. Imposing the conditions $\eta(\mathbf{L},\mathbf{e}_i)=0$ and $\eta(\mathbf{L},\mathbf{X})=1$ leads to
     	\[
     	\eta\Big(\mathbf{L},\frac{\partial}{\partial \theta^i}\Big) = -\frac{r_i}{r}.
     	\]  
     	This gives rise to a non-homogeneous linear system of equations $A\mathbf{x}=\mathbf{b}$, where $\mathbf{x} = (a_2,\dots,a_{n+2})^\mathrm{T}$, and the vector $\mathbf{b}$ is
     	\[
     	\mathbf{b} = \Big(a+\frac{1}{r}, -\frac{r_1}{r^2},\dots,-\frac{r_n}{r^2}\Big)^\mathrm{T}.
     	\]  
     	Each row of the $(n+1)$-order matrix $A$ can be explicitly written as  
     	       \begin{align*}
     		A_1\;=\;&\Big(\cos\theta^1,\;\sin\theta^1\cos\theta^2,\; \sin\theta^1\sin\theta^2\cos\theta^3,\; \cdots,\; \\
     		&\qquad\qquad\qquad \sin\theta^1\cdots\sin\theta^{n-1}\cos\theta^n,\; \sin\theta^1\cdots\sin\theta^{n-1}\sin\theta^n \Big),\;\\
     		A_2\;=\;&\Big(-\sin\theta^1,\; \cos\theta^1\cos\theta^2,\; \cos\theta^1\sin\theta^2\cos\theta^3,\; \cdots,\; \\
     		&\qquad\cos\theta^1\sin\theta^2\cdots\sin\theta^{n-1}\cos\theta^n,\; \cos\theta^1\sin\theta^2\cdots\sin\theta^{n-1}\sin\theta^n\Big),\; \\
     		A_3\;=\;&\Big(0,\;-\sin\theta^1\sin\theta^2,\; \sin\theta^1\cos\theta^2\cos\theta^3,\; \cdots,\; \\
     		&\qquad\qquad\qquad \sin\theta^1\cos\theta^2\sin\theta^3\cdots\sin\theta^{n-1}\cos\theta^n,\; \\
     		&\qquad\qquad\qquad\qquad\sin\theta^1\cos\theta^2\sin\theta^3\cdots\sin\theta^{n-1}\sin\theta^n \Big),\;\\
     		\;\vdots&\\
     		A_{n+1}\;=\;&\Big(0,\; 0,\; 0,\; \cdots,\; -\sin\theta^1\cdots\sin\theta^n,\; \sin\theta^1\cdots\sin\theta^{n-1}\cos\theta^n \Big),\;
     	\end{align*}
     	A direct computation yields the inverse matrix $A^{-1}$, whose rows are given by  
     	 \begin{align*}
     		B_1\;=\;&\Big(\cos\theta^1,\;-\sin\theta^1,\; 0,\; \cdots,\;  0,\; 0 \Big),\;\\
     		B_2\;=\;&\Big(\sin\theta^1\cos\theta^2,\;\cos\theta^1\cos\theta^2,\; -\frac{\sin\theta^2}{\sin\theta^1},\; 0,\; \cdots,\; 0,\; 0\Big),\;\\
     		B_3\;=\;&\Big(\sin\theta^1\sin\theta^2\cos\theta^3,\;\cos\theta^1\sin\theta^2\cos\theta^3,\;   \\
     		&\qquad\qquad\qquad \frac{\cos\theta^2\cos\theta^3}{\sin\theta^1},\; -\frac{\sin\theta^3}{\sin\theta^1\sin\theta^2},\;0,\;\cdots,\; 0,\;0 \Big),\;\\
     		\;\vdots&\\
     		B_{n+1}\;=\;&\Big(\sin\theta^1\cdots\sin\theta^n,\;\cos\theta^1\sin\theta^2\cdots\sin\theta^n,\;  \frac{\cos\theta^2\sin\theta^3\cdots\sin\theta^n}{\sin\theta^1},\; \\
     		&\qquad\qquad\qquad  \frac{\cos\theta^3\sin\theta^4\cdots\sin\theta^n}{\sin\theta^1\sin\theta^2},\;\cdots,\; \frac{\cos\theta^n}{\sin\theta^1\cdots\sin\theta^{n-1}} \Big).
     	\end{align*}
     	Solving for $\mathbf{x}$ yields $\mathbf{x} = A^{-1}\mathbf{b}$, providing explicit expressions for the components of $\mathbf{L}$
     	 \begin{align*}
     		a_2\;=\;&a\cos\theta^1+\frac{\left(r\cos\theta^1+r_1\sin\theta^1\right)}{r^2},\\
     		a_3\;=\;&a\sin\theta^1\cos\theta^2\\
     		&\;+\frac{\left(r\sin^2\theta^1\cos\theta^2-r_1\sin\theta^1\cos\theta^1\cos\theta^2+r_2\sin\theta^2\right)}{r^2\sin\theta^1},\\
     		a_4\;=\;&a\sin\theta^1\sin\theta^2\cos\theta^3\\
     		&\;+\frac{\left(r\sin^2\theta^1\sin^2\theta^2\cos\theta^3-r_1\sin\theta^1\cos\theta^1\sin^2\theta^2\cos\theta^3\right)}{r^2\sin\theta^1\sin\theta^2}\\
     		&\;\quad+\frac{\left(-r_2\sin\theta^2\cos\theta^2\cos\theta^3+r_3\sin\theta^3\right)}{r^2\sin\theta^1\sin\theta^2},\\
     		\;\vdots&\\
     		a_{n+2}\;=\;&a\sin\theta^1\cdots\sin\theta^n\\
     		&\;+\frac{r\sin^2\theta^1\cdots\sin^2\theta^{n-1}\cos\theta^n}{r^2\sin\theta^1\cdots\sin\theta^{n-1}}\\
     		&\;\;+\frac{\left(-r_1\sin\theta^1\cos\theta^1\sin^2\theta^2\cdots\sin^2\theta^{n-1}\sin\theta^n-\cdots-r_n\cos\theta^n\right)}{r^2\sin\theta^1\cdots\sin\theta^{n-1}}.\\
     	\end{align*}       	
     	Finally, imposing the null condition $-a^2 + a_2^2 + \cdots + a_{n+2}^2 = 0$ yields
     	\[
     	a = -\frac{1}{2r}\Big(1 + |Dr|^2\Big),
     	\]  
     	which completes the proof.
     \end{proof}
\begin{cor}
	The vector $\mathbf{L}$ can be decomposed as
	\[
	\mathbf{L} = \frac{a}{r}\mathbf{X} + \mathbf{Z},
	\]
	where
	\[
	\mathbf{Z} = \left(0,\, a_2 - a\cos\theta^1,\, a_3 - a\sin\theta^1\cos\theta^2,\, \dots,\, a_{n+2} - a\sin\theta^1\cdots\sin\theta^n\right),
	\]
	which satisfies
	\[
	\eta(\mathbf{X}, \mathbf{Z}) = 1, \qquad \eta(\mathbf{Z}, \mathbf{Z}) = -\tfrac{2a}{r}.
	\]
\end{cor}
     With the normal frame $\{\mathbf{X}, \mathbf{L}\}$, the structure equations take the form
     \begin{gather}\label{eqn-stru}
     	\begin{split}
     		\bar{\nabla}_i\mathbf{e}_j &= \nabla_i\mathbf{e}_j - g_{ij} \mathbf{L} + h_{ij} \mathbf{X},\\
     		\bar{\nabla}_i\mathbf{X} &= \mathbf{e}_i, \quad \bar{\nabla}_i\mathbf{L} = - h_i^j \mathbf{e}_j.
     	\end{split}
     \end{gather}  
     where $\bar{\nabla}$ denotes the flat connection of the ambient space $\E^{n+2}_1$, $\nabla$ is the Levi-Civita connection on $M^n$ with respect to the metric $g$, and $h_{ij}$ represent the components of the second fundamental form. Note that throughout this paper, we adopt the following conventions for brevity and consistency: The shorthand notation $\nabla_i$ is used to denote  covariant derivative $\nabla_{\mathbf{e}_i}$.
     Subscripts following a comma (e.g., $r_{i,j}$) signify  covariant differentiation with respect to the metric $g_{ij}$. Indices are raised and lowered using the metric $g_{ij}$ and its inverse $g^{ij}$, respectively, unless explicitly stated otherwise.
     
     A straightforward computation shows that the second fundamental form is given by
     \begin{align*}
	h_{ij}=-\frac{a}{r}g_{ij}-\mu_{ij}-\nu_{ij}=\frac{a}{r}g_{ij}+\frac{r_{i,j}}{r},
\end{align*}
     where
     \begin{align*}
     	\mu_{ij} &= \frac{g_{ij}}{r^2} + \frac{2 r_i r_j}{r^2} - \frac{r_{ij}}{r},\\
     	\nu_{ij} &= \frac{r_j \cos\theta^i}{r \sin\theta^i}, \quad i<j,\\
     	\nu_{ii} &= -\frac{1}{2r} \sum_{k=1}^{i-1} r_k \sin(2\theta^k) \prod_{s=k+1}^{i-1} \sin^2\theta^s.
     \end{align*}  
     Consequently, the mixed form and trace of $h$ are
     \begin{align}\label{kappa}
     	h_i^j = \frac{a}{r} \delta_i^j + \frac{r^j_{,i}}{r}, \qquad g^{ij} h_{ij} = \frac{n a}{r} + \frac{\Delta_g r}{r}.
     \end{align}
     The curvature tensor and its contractions are given by
          \begin{align*}
     	R^l_{jki}&\;=\;\frac{\partial\Gamma^l_{ij}}{\partial\theta^k}-\frac{\partial\Gamma^l_{kj}}{\partial\theta^i}+\Gamma^m_{ij}\Gamma^l_{mk}-\Gamma^m_{kj}\Gamma^l_{mi}\\
     	&\;=\;g_{kj}h^l_i+h_{kj}\delta^l_i-g_{ij}h^l_k-h_{ij}\delta^l_k.
     \end{align*}
     \begin{align*}
     	R_{ji}=R^{l}_{jli}=\left(2-n\right)h_{ij}-g_{ij}h^l_l.
     \end{align*}
     \begin{align*}
     	R=g^{ij}R_{ij}=2(1-n)h^l_l.
     \end{align*}
     Moreover, the Codazzi symmetry holds
     \[
     h_{ij,k} = h_{ik,j}.
     \]
    
%     \begin{gather*}
%       \chi_{11}=1+\frac{3r_1^2}{r^2}-\frac{2r_{11}}{r}-\frac{r_2^2}{r^2\sin^2\theta^1},\\
%       \chi_{12}=\frac{4r_1r_2}{r^2}-\frac{2r_{12}}{r}+\frac{2r_1r_2\cos\theta^1}{r^2\sin\theta^1},\\
%       \chi_{22}=\sin^2\theta^1+\frac{3r_2^2}{r^2}-\frac{2r_{22}}{r}-\frac{r_1^2\sin^2\theta^1}{r^2}-\frac{2r_1\sin\theta^1\cos\theta^1}{r}.
%     \end{gather*}
\section{Some properties}\label{sec-prop}
Let $S_r$ denote the $r$-th elementary symmetric function of the eigenvalues $k_1,\dots,k_n$ of the matrix $\Phi=(h^i_j)$, namely,
\[
S_0=1,\quad S_1=k_1+\cdots+k_n,\quad \dots,\quad S_n=k_1\cdots k_n.
\]
The Newton transformations $T_r$ are defined recursively by
\[
T_0{}^i_j=\delta^i_j,\qquad
T_{r+1}{}^i_j = S_{r+1}\delta^i_j - T_r{}^{ik} h_{kj},\quad r=0,1,\cdots,n-1.
\]
Using the generalized Kronecker symbol ${\epsilon_{i_1\cdots i_q}}^{j_1\cdots j_q}$, 
$T_r$ can be written as
\[
{T_r}^j_i = \frac{1}{r!}\,{\epsilon_{i_1\cdots i_ri}}^{j_1\cdots j_rj}\,
h_{j_1}^{i_1}\cdots h_{j_r}^{i_r},
\]
where ${\epsilon_{i_1\cdots i_q}}^{j_1\cdots j_q}$ equals $+1$ (even permutation), $-1$ (odd permutation), or $0$ otherwise.
The Newton transformations satisfy the following properties. For details, see \cite{rei}.
\begin{prop}\label{N prop}
	For each $r = 0, 1, \dots, n-1$, the following properties hold:
	\begin{itemize}[left=12pt, labelsep=5pt, itemsep=-4pt]
		\item[$\mathrm{(i)}$] $T_r \Phi = \Phi T_r$, and $T_n = 0$.\\
		\item[$\mathrm{(ii)}$] \emph{Newton's formula:}
		\[
		(r+1) S_{r+1} = \mathrm{Trace}(\Phi T_r).
		\]
		\item[$\mathrm{(iii)}$] If $\Phi = \Phi(t)$ depends smoothly on a parameter $t$, then
		\[
		\frac{d}{dt} S_{r+1} = \mathrm{Trace}\Big(\frac{d\Phi}{dt} \cdot T_r\Big).
		\]
		\item[$\mathrm{(iv)}$] The Newton tensors are divergence-free
		\[
		{{T_r}^{ij}}_{,j} = 0.
		\]
	\end{itemize}
\end{prop}
The following proposition characterizes hypersurfaces lying in a hyperplane in terms of their shape operator.
\begin{prop}\label{prop-plane}
	The hypersurface $M^n$ lies on a hyperplane $P(\mathbf{v},1)$ in $\E^{n+2}_1$ if and only if
	\begin{align*}
		h^j_i=c\delta^j_i,
	\end{align*}
	where $c$ is a constant.
\end{prop}

\begin{proof}
	If the hypersurface $M^n$ lies on a hyperplane $P(\mathbf{v},1)$, then
	\begin{align*}
		\eta(\mathbf{X},\mathbf{v})=1, \qquad \eta(\mathbf{e}_i,\mathbf{v})=0.
	\end{align*}
	By Proposition \ref{prop-unique}, we have $\mathbf{L}=-\tfrac{\eta(\mathbf{v},\mathbf{v})}{2}\mathbf{X}+\mathbf{v}.$
	Then from \eqref{eqn-stru}, it follows that
	\begin{align*}
		h^j_i=\frac{\eta(\mathbf{v},\mathbf{v})}{2}\delta^j_i.
	\end{align*}
	Conversely, if $h^j_i=c\delta^j_i$, then by \eqref{eqn-stru} one finds $\mathbf{L}=-c\mathbf{X}+\mathbf{v},$
	where $\mathbf{v}$ is a constant vector. This implies $\eta(\mathbf{X},\mathbf{v})=1$, hence $M^n$ lies in the hyperplane $P(\mathbf{v},1)$.
\end{proof}
%\begin{align*}
%  R^i_j=(2-n)h^i_j-\delta^i_jh^l_l,\quad R^i_{j,i}=(1-n)h^i_{i,j}=\frac{1}{2}R_j.
%\end{align*}
%
%\begin{align*}
%  t_{ij}=\frac{1}{2}R\delta_{ij}-R_{ij}, \quad\mathrm{div}(T)=0.
%\end{align*}
If we denote
\begin{align*}
	S=h^i_jh^j_i,\qquad f_3=h^i_jh^j_kh^k_i,
\end{align*}
then the following relations hold
\begin{align*}
	\Delta h_{ij}&=\left(S_1\right)_{,ij}+g_{ij}S-nh_{ik}h^k_j,\\
	\tfrac{1}{2}\Delta S&=\left(S_1\right)^j_ih^i_j-nf_3+S_1S
	+g^{li}g^{mj}g^{pk}h_{lm,p}h_{ij,k},\\
	h_{ij,kl}-h_{ij,lk}&=g_{kj}h_{im}h^m_l-g_{jl}h_{im}h^m_k
	+g_{ik}h_{jm}h^m_l-g_{il}h_{jm}h^m_k.
\end{align*}
As a direct consequence, we obtain the following integral relation for $1$-extremal hypersurfaces.
\begin{prop}
	Let $\mathbf{X}:M^n\rightarrow LC^*$ be a closed orientable hypersurface with $S_1=0$. Then
	\begin{align*}
		n\int_M f_3\,dV=\int_M g^{li}g^{mj}g^{pk}h_{lm,p}h_{ij,k}\,dV \;\geq\;0. 
	\end{align*}
\end{prop}

\section{Evolution Equations}\label{sec-evo}
Consider a smooth family of immersions $\mathbf{X}: M^n \times [0,T] \to LC^*$
satisfying
\begin{equation}\label{flow}
	\dot{\mathbf{X}}(\theta,t) = U(\theta,t) \mathbf{X}(\theta,t) + W^i(\theta,t) \mathbf{e}_i(\theta,t), \quad \theta = \{\theta^1, \dots, \theta^n\},
\end{equation}
with initial immersion $\mathbf{X}_0 = \mathbf{X}(p,0)$, where $\displaystyle\dot{\mathbf{X}} = \frac{dr}{dt} \frac{\partial}{\partial r} + \frac{d\theta^i}{dt} \frac{\partial}{\partial \theta^i}.$
We first derive the evolution equations for the basic geometric quantities.
\begin{prop}
	Under the flow \eqref{flow}, the following evolution equations hold:
	\begin{enumerate}[leftmargin=20pt]
		\item [\textnormal{(1)}] \quad$\displaystyle\frac{d}{dt} g_{ij} = 2 U g_{ij} + W_{i,j} + W_{j,i}.$
		\item [\textnormal{(2)}] \quad$\displaystyle\frac{d}{dt} h_{ij} = U_{,ij} + W^k_{,j} h_{ki} + W^k_{,i} h_{kj} + W^k h_{kj,i}.$
		\item [\textnormal{(3)}] \quad$\displaystyle\frac{d}{dt} h_i^k = U_{,is} g^{sk} - 2 U h_i^k + W^s_{,i} h_s^k + W^s h_{s,i}^k - h_i^s W^k_{,s}.$
	\end{enumerate}
\end{prop}
\begin{proof}
	(1) Since
	\[
	\bar{\nabla}_{\dot{\mathbf{X}}}\mathbf{e}_i - \bar{\nabla}_{\mathbf{e}_i}\dot{\mathbf{X}}
	= \mathbf{X}_*\left(\left[\frac{\partial}{\partial t}, \frac{\partial}{\partial \theta^i}\right]\right)= 0,
	\]
	we obtain
	\begin{align*}
		\frac{d}{dt} g_{ij}
		&= \bar{\nabla}_{\frac{\partial}{\partial t}} \eta(\mathbf{e}_i,\mathbf{e}_j) \\
		&= \eta(\bar{\nabla}_{\dot{\mathbf{X}}}\mathbf{e}_i,\mathbf{e}_j)
		+ \eta(\bar{\nabla}_{\dot{\mathbf{X}}}\mathbf{e}_j,\mathbf{e}_i) \\
		&= \eta(\bar{\nabla}_{\mathbf{e}_i}\dot{\mathbf{X}},\mathbf{e}_j)
		+ \eta(\bar{\nabla}_{\mathbf{e}_j}\dot{\mathbf{X}},\mathbf{e}_i) \\
		&= \eta(\bar{\nabla}_{\mathbf{e}_i}(U\mathbf{X}+W^k\mathbf{e}_k),\mathbf{e}_j)
		+ \eta(\bar{\nabla}_{\mathbf{e}_j}(U\mathbf{X}+W^k\mathbf{e}_k),\mathbf{e}_i) \\
		&= 2U g_{ij} + W_{i,j} + W_{j,i}.
	\end{align*}
	(2) For the null normal $\mathbf{L}$ one has
	\[
	\frac{d}{dt}\mathbf{L} = -U\mathbf{L} - (W^k h_{ki} + \bar{\nabla}_{\mathbf{e}_i}U) g^{ij}\mathbf{e}_j.
	\]
	Hence,
	\begin{align*}
		\frac{d}{dt} h_{ij}
		&= \eta(\bar{\nabla}_{\dot{\mathbf{X}}}\bar{\nabla}_{\mathbf{e}_i}\mathbf{e}_j,\mathbf{L})
		+ \eta(\bar{\nabla}_{\mathbf{e}_i}\mathbf{e}_j,\bar{\nabla}_{\dot{\mathbf{X}}}\mathbf{L}) \\
		&= \eta(\bar{\nabla}_{\mathbf{e}_i}\bar{\nabla}_{\mathbf{e}_j}\dot{\mathbf{X}},\mathbf{L})
		+ \eta(\bar{R}(\dot{\mathbf{X}},\mathbf{e}_i)\mathbf{e}_j,\mathbf{L}) \\
		&\quad + \eta(\Gamma^k_{ij}\mathbf{e}_k - g_{ij}\mathbf{X} + h_{ij}\mathbf{L},
		-U\mathbf{L} - (W^k h_{kl} + \bar{\nabla}_{\mathbf{e}_l}U) g^{lm}\mathbf{e}_m) \\
		&= U_{,ij} + W^k_{,j} h_{ki} + W^k_{,i} h_{kj} + W^k h_{kj,i}.
	\end{align*}
	(3) For the mixed form we compute
	\begin{align*}
		\frac{d}{dt} h_i^k
		&= \frac{d}{dt}(g^{kj} h_{ij}) \\
		&= h_{ij}\frac{d}{dt} g^{kj} + g^{kj}\frac{d}{dt} h_{ij} \\
		&= -2U h_i^k - h_i^j W_{,j}^k - h_{ij} g^{kl} W_{,l}^j \\
		&\quad + g^{kj} U_{,ij} + g^{kj} W^l_{,j} h_{li} + W^l_{,i} h_l^k + W^l h_{l,i}^k \\
		&= U_{,is} g^{sk} - 2U h_i^k + W^s_{,i} h_s^k + W^s h_{s,i}^k - h_i^s W^k_{,s}.
	\end{align*}
\end{proof}
Using the properties of Newton transformations, we can further derive the evolution equation of  $S_{r+1}$.
\begin{lem}\label{lem-evo-S}
	The evolution of $S_{r+1}$ is given by
	\[
	\frac{d}{dt} S_{r+1} = (T_r)^{ij} U_{,ij} - 2(r+1) S_{r+1} U + {S_{r+1}}_{,s} W^s.
	\]
\end{lem}
\begin{proof}
	By Proposition~\ref{N prop}(iii), we have
	\[
	\frac{d}{dt} S_{r+1} 
	= (T_r)^i_k \big(U_{,is} g^{sk} - 2U h_i^k + W^s_{,i} h_s^k + W^s h_{s,i}^k - h_i^s W^k_{,s}\big).
	\]
	From Proposition~\ref{N prop}(i), the terms involving derivatives of $W$ cancel
	\[
	(T_r)^i_k \big(W^s_{,i} h_s^k - h_i^s W^k_{,s}\big) = 0.
	\]
	Moreover, applying Proposition~\ref{N prop}(ii) and (iv) yields
	\[
	(T_r)^i_k W^s h_{s,i}^k = {S_{r+1}}_{,s} W^s, 
	\qquad (T_r)^i_k h_i^k = (r+1) S_{r+1}.
	\]
	Substituting these identities back into the expression for $\frac{d}{dt} S_{r+1}$ immediately gives the claimed formula.
\end{proof}
We now proceed to prove Theorem~\ref{main thm A}
\begin{proof}[Proof of Theorem~\ref{main thm A}]
	The application of Lemma \ref{lem-evo-S} in conjunction with integration by parts leads to the results outlined in Theorem~\ref{main thm A}.
\end{proof}
When $n = 2k$ is an even integer, the integral $\int_M S_kdV$
can be reduced to a functional depending solely on $r$. To this end, we first show that $S_k$ can be decomposed into two terms, one of which is precisely the divergence of a certain tangential vector field. In the case $n=2$, since
\begin{align*}
	S_1 = \frac{n a}{r} + \frac{1}{r} r^i_{,i} = -\frac{1}{r^2} + \left(\log r\right)^i_{,i},
\end{align*}
the claim immediately follows. For $n=4$, noting that
\[
S = h^j_i h^i_j = \frac{n a^2}{r^2} + \frac{2a}{r} r^i_{,i} + \frac{r^i_j r^j_i}{r^2},
\]
we compute
\begin{align*}
	2 S_2 &= S_1^2 - S \\
	&= \frac{n(n-1)}{4 r^4} + \frac{n(n-1)}{2 r^4} |Dr|^2 + \frac{n(n-1)}{4 r^4} |Dr|^4 - \frac{n-1}{r^3} r^i_{,i} \\
	&\quad - \frac{n-1}{r^3} |Dr|^2 r^i_{,i} + \frac{(r^i_{,i})^2 - r^i_{,j} r^j_{,i}}{r^2}.
\end{align*}
Moreover, the following identities hold
\begin{align*}
	\frac{1}{2} g^{ml} \left( \frac{g^{ij} r_i r_j}{r^2} \right)_{,ml} &= \frac{3}{r^4} |Dr|^4 - \frac{4}{r^3} g^{ij} g^{ml} r_{,im} r_j r_l\\
	&\qquad- \frac{r^i_{,i}}{r^3} |Dr|^2 + \frac{1}{r^2} r^i_{,li} r^l + \frac{1}{r^2} r^i_{,j} r^j_{,i}, \\
	-2 g^{lm} \left( \frac{1}{r^2} g^{ij} r_{,il} r_j \right)_{,m} &= \frac{4}{r^3} g^{lm} g^{ij} r_{,il} r_j r_m - \frac{2}{r^2} r^i_{,li} r^l - \frac{2}{r^2} r^i_{,j} r^j_{,i}, \\
	\left( \frac{r^i_{,i} r^l}{r^2} \right)_{,l} &= \frac{(r^i_{,i})^2}{r^2} - \frac{2}{r^3} r^i_{,i} |Dr|^2 + \frac{r^l}{r^2} r^i_{,il}, \\
	\frac{r^l}{r^2} \left( r^i_{,il} - r^i_{,li} \right) &= -\frac{n-1}{r^4} |Dr|^2 - \frac{n-1}{r^4} |Dr|^4\\
	&\qquad+ \frac{n-2}{r^3} r_{,ml} r^m r^l + \frac{1}{r^3} r^i_{,i} |Dr|^2, \\
	\left( \frac{r^m r^l r_l}{r^3} \right)_{,m} &= -\frac{3}{r^4} |Dr|^4 + \frac{r^m_{,m}}{r^3} |Dr|^2 + \frac{2 r^m r^l r_{,lm}}{r^3}.
\end{align*}
Combining these computations, we arrive at
\begin{align*}
	S_2 = \frac{3}{2 r^4} + \Delta \left( \frac{1 + |Dr|^2}{4 r^2} \right) + \left( \frac{r^i_{,i} r^m - 2 r^m_{,i} r^i - |Dr|^2 (\log r)_{,}^m}{2 r^2} \right)_{,m}.
\end{align*}
Continuing similar calculations and using Corollary \ref{1 var} for simplification along the way, we obtain the following general result.
\begin{prop}
	Let $\mathbf{X}:M^n\rightarrow LC^*$  represent a hypersurface and  $n=2k$ is  an even integer. Then, the following identity holds:
	\begin{align*}
		S_k=\frac{(-2)^{-k}}{kr^n}\frac{n!}{k!}+\mathrm{div}(\mathbf{V}),
	\end{align*}
	where, for different values of $k$, 
	$\mathbf{V}$ denotes specific tangent vector fields defined on the hypersurface $M$.
\end{prop}
\begin{cor}
	Let $\mathbf{X}:M^n\rightarrow LC^*$  be a closed orientable hypersurface and  $n=2k$ is an even integer. Then,
	\begin{align*}
		\int_MS_kdV=\frac{(-2)^{-k}}{k}\frac{n!}{k!}\int_M\frac{1}{r^n}dV.
	\end{align*}
\end{cor}

%     \begin{thm}
%       Suppose that $f$ is any smooth function of $n$ variable. Then
%       \begin{align*}
%         \frac{d}{dt}\int_Mf(S_1,\cdots,S_n)dV=\int_MU\left(\sum_{r=1}^n\left(D_rf\right)_{,ij}\left(T_{r-1}\right)^{ij}+nf-2rS_r\right)dV.
%       \end{align*}
%     \end{thm}
%     \begin{thm}
%       Suppose that . Then
%       \begin{align*}
%          \frac{d}{dt}\int_MS_rdV=\int_M U(n-2r)S_rdV.
%       \end{align*}
%       If, in addition, the immersion yields an extremal value for $n\neq 2r$, i.e., $S_r=0$ when $t=t_0$, then
%       \begin{align*}
%         \frac{d^2}{dt^2}=(2r-n)\int_M\left(T_{r-1}\right)^{ij}U_iU_jdV.
%       \end{align*}
%     \end{thm}
%\newpage

%\newpage

\section{a length-preserving curve flow}\label{sec-flow}
In this section, we study a length-preserving curve flow on the $2$-dimensional light-cone $LC^*$.
Consider the flow
\begin{equation}\label{length pre flow}
	\left\{
	\begin{aligned}
		&\frac{\partial\mathbf{X}}{\partial t}=\left(r^{\frac{1}{2}}\kappa+\frac{1}{2}r^{-\frac{3}{2}}\right)\mathbf{X},\\
		&\mathbf{X}(\cdot,0)=\mathbf{X}_0(\cdot),
	\end{aligned}
	\right.
\end{equation}
where $\mathbf{X}: \S^1\times[0,T)$ is a family of smooth, closed, space-like curves, $r$ is the  radial function, and $\kappa$ denotes the curvature of the curve on the light-cone $LC^*$. 
Let $\bar{g}:=\frac{\d s}{\d p}=\sqrt{\langle\frac{\partial\mathbf{X}}{\partial p},\frac{\partial\mathbf{X}}{\partial p}\rangle_L}$ be the induced metric of the evolving curve.  Under the flow \eqref{length pre flow}, $\bar{g}$ evolves according to
\begin{align}\label{g-_t}
	\frac{\partial\bar{g}}{\partial t}=\left(r^{\frac{1}{2}}\kappa+\frac{1}{2}r^{-\frac{3}{2}}\right)\bar{g}. 
\end{align}
The following proposition demostrates that the curve length $\mathcal{L}(t):=\oint\d s$ remains constant along this flow.
\begin{prop}\label{L}
	Under the flow \eqref{length pre flow}, the length $\mathcal{L}(t)$ of the evolving curve $\mathbf{X}(\cdot,t)$ satisfy that
	\begin{align*}
		\frac{\d\mathcal{L}}{\d t}=0.
	\end{align*}
\end{prop}
\begin{proof}
	Using \eqref{g-_t} and \eqref{kappa}, we compute
	\begin{align*}
		\frac{\d\mathcal{L}}{\d t}&=\oint\frac{1}{\bar{g}}\frac{\partial \bar{g}}{\partial t}\d s\\
		&=\oint\left(r^{-\frac{1}{2}}r_{ss}-\frac{1}{2}r_s^2r^{-\frac{3}{2}}\right)\d s\\
		&=0.
	\end{align*}
\end{proof}
To study the existence and regularity of the flow \eqref{length pre flow}, we reduce it to an equivalent scalar equation. Consider a general curve flow on the light-cone $LC^*$ of the form
\begin{align}\label{general flow}
	\frac{\partial \mathbf{X}}{\partial t}(p,t)=U(p,t)\mathbf{X}(p,t) +W(p,t)\mathbf{T}(p,t),
\end{align}
where $\mathbf{X}: \S^1\times[0,T)$ is a family of smooth, closed, space-like curves, $\mathbf{T}$ is the unit tangent vector, and $U$, $W$ are smooth functions.
Utilizing the polar coordinate system $(r,\theta)$, the evolving curve can be represented as
\begin{align}\label{evo cur}
	\mathbf{X}(\theta,t)=r(\theta,t)\mathbf{P}(\theta),
\end{align}
where $\mathbf{P}(\theta)=(1,\cos\theta,\sin\theta)$.
	 Let $\mathbf{Q}(\theta)=(0,-\sin\theta,\cos\theta)$, and note that $\frac{\d s}{\d\theta}=r$. Then the unit tangent vector of the curve can be written as
\begin{align}\label{tangent vector}
	\mathbf{T}=\frac{1}{ r}\frac{\partial r}{\partial\theta}\mathbf{P}+\mathbf{Q}.
\end{align}
When the curve evolves under flow \eqref{general flow}, we have
\begin{align*}
	U\mathbf{X}+W\mathbf{T}=\frac{\partial r}{\partial t}\mathbf{P}+ r \frac{\partial\theta}{\partial t}\mathbf{Q}.
\end{align*}
Substituting \eqref{tangent vector} into the above and comparing coefficients, we obtain the following evolution equations
\begin{align*}
	\frac{\partial r}{\partial t}=U r+\frac{W}{ r}\frac{\partial r}{\partial\theta}, \qquad	\frac{\partial\theta}{\partial t}=\frac{W}{ r}.
\end{align*}
Since the tangential component affects only the parameterization of the curve, we may take $W=0$,  in which case $\frac{\partial\theta}{\partial t}\equiv0$, implying that the polar angle $\theta$ is independent of the time $t$.
The evolution equation for the radial function then becomes
\begin{align}\label{psi evo}
	\frac{\partial r}{\partial t}=U r.
\end{align}
If $ r= r(\theta,t)>0$ is defined on $[0,2\pi]\times[0,T)$ and satisﬁes the equation \eqref{psi evo}, then the family of curves $\{\mathbf{X}= r\mathbf{P}\vert t\in[0,T)\}$ solves the flow \eqref{general flow}. This allows us to reduce the study of \eqref{general flow} to the Cauchy problem $\eqref{psi evo}$ for the scalar function $ r$ with positive initial value $ r_0(\theta)>0$. 
\begin{lem}\label{equi}
	Suppose the initial curve $\mathbf{X}_0$ is a smooth, closed, space-like curve lying on the light-cone $LC^*$. Then, on some time interval, the flow \eqref{general flow} is equivalent to the evolution equation $\eqref{psi evo}$ with positive initial value $ r_0$.
\end{lem}
According to Lemma \ref{equi}, the flow \eqref{length pre flow} is equivalent to the following initial value problem
\begin{equation}\label{reduce}
	\left\{
	\begin{aligned}
		&\frac{\partial r}{\partial t}= r^{-\frac{3}{2}}\frac{\partial^2 r}{\partial\theta^2}-\frac{3}{2} r^{-\frac{5}{2}}\left(\frac{\partial r}{\partial\theta}\right)^2,\\
		& r(\theta,0)= r_0(\theta),
	\end{aligned}
	\right.
\end{equation}
where $ r_0$ is the radial function corresponding to the initial space-like curve.

\subsection{Estimates and longtime existence}
In order to show that the flow \eqref{length pre flow} exists for all times, we need to establish several a priori estimates. We begin by deriving the $C^0$ estimate.  
%For simplicity, we adopt a compact notation throughout, such as writing $ r_\theta = \frac{\partial r}{\partial\theta}$.
\begin{lem}\label{C0}
	Suppose $ r(\theta,t)$ is a solution of the initial value problem \eqref{reduce}. Then for any $(\theta,t)\in[0,2\pi]\times[0,T]$, the following estimate holds
	\begin{align*}
		\min\limits_{\theta\in[0,2\pi]} r(\theta,0)\le r(\theta,t)\le\max\limits_{\theta\in[0,2\pi]} r(\theta,0).
	\end{align*}
\end{lem}
\begin{proof}
	This follows from the maximum principle applied to the evolution of the radial function 
	\begin{align}\label{0order}
		\frac{\partial r}{\partial t}= r^{-\frac{3}{2}} r_{\theta\theta}-\frac{3}{2} r^{-\frac{5}{2}} r_\theta^2.
	\end{align}
\end{proof}
Next, we establish the gradient estimate, for which the following lemma is needed.
\begin{lem}\label{rho evo}
	Along \eqref{reduce}, the quantity 
	\begin{align*}
		\rho:=\frac{1}{2} r_\theta^2
	\end{align*}
	satisfies the evolution equation
	\begin{align*}
		\frac{\partial\rho}{\partial t}= r^{-\frac{3}{2}}\rho_{\theta\theta}- r^{-\frac{3}{2}} r_{\theta\theta}^2-\frac{9}{2} r^{-\frac{5}{2}} r_\theta\rho_\theta+15 r^{-\frac{7}{2}}\rho^2.
	\end{align*}
\end{lem}
\begin{proof}
	Differentiating $\eqref{0order}$ yields
	\begin{align}\label{1order}
		\frac{\partial r_\theta}{\partial t}= r^{-\frac{3}{2}} r_{\theta\theta\theta}-\frac{9}{2} r^{-\frac{5}{2}} r_\theta r_{\theta\theta}+\frac{15}{4} r^{-\frac{7}{2}} r_\theta^3.
	\end{align}
	From this, we compute
	\begin{align*}
		\frac{\partial\rho}{\partial t}&= r^{-\frac{3}{2}} r_\theta r_{\theta\theta\theta}-\frac{9}{2} r^{-\frac{5}{2}} r_\theta^2 r_{\theta\theta}+\frac{15}{4} r^{-\frac{7}{2}} r_\theta^4\\
		&= r^{-\frac{3}{2}}\rho_{\theta\theta}- r^{-\frac{3}{2}} r_{\theta\theta}^2-\frac{9}{2} r^{-\frac{5}{2}} r_\theta\rho_\theta+15 r^{-\frac{7}{2}}\rho^2.
	\end{align*}
\end{proof}
We now derive the uniform gradient estimate.
\begin{lem}\label{C1}
	Suppose $ r(\theta,t)$ is a solution of the initial value problem \eqref{reduce}.  Then for any  $(\theta,t)\in[0,2\pi]\times[0,T]$, there exists a constant $C$, depending only on the initial data $ r_0$, such that
	\begin{align*}
		\lvert r_\theta(\theta,t)\rvert\le C.
	\end{align*}
\end{lem}
\begin{proof}
	At points where $\rho>0$, define the auxiliary function
	\begin{align*}
		\mathcal{Q}:=\log\rho-4\log r.
	\end{align*}
	Then a straightforward computation gives
	\begin{align}\label{Q1}
		\mathcal{Q}_\theta=\rho^{-1}\rho_\theta-4 r^{-1} r_\theta,
	\end{align}
	and
	\begin{align}\label{Q2}
		\mathcal{Q}_{\theta\theta}=\rho^{-1}\rho_{\theta\theta}-\rho^{-2}\rho_\theta^2+8 r^{-2}\rho-4 r^{-1} r_{\theta\theta}.
	\end{align}
	Using \eqref{0order} and Lemma \ref{rho evo}, we find the evolution of $\mathcal{Q}$ 
	\begin{align*}
		\frac{\partial\mathcal{Q}}{\partial t}= r^{-\frac{3}{2}}\rho^{-1}\rho_{\theta\theta}- r^{-\frac{3}{2}}\rho^{-1} r_{\theta\theta}^2-9 r^{-\frac{5}{2}} r_{\theta}^{-1}\rho_\theta+27 r^{-\frac{7}{2}}\rho-4 r^{-\frac{5}{2}} r_{\theta\theta}.
	\end{align*}
	Substituting \eqref{Q1} and \eqref{Q2}, we arrive at
	\begin{align*}
		\frac{\partial\mathcal{Q}}{\partial t}= r^{-\frac{3}{2}}\mathcal{Q}_{\theta\theta}+\frac{1}{2} r^{-\frac{3}{2}}\mathcal{Q}_\theta^2-\frac{1}{2} r^{-\frac{5}{2}} r_{\theta}\mathcal{Q}_\theta- r^{-\frac{7}{2}}\rho.
	\end{align*}
	Assume that the maximum of $\mathcal{Q}$ over $[0, 2\pi] \times [0, T]$ is attained for the first time at $(\theta_0, t_0)$ with $t_0 > 0$. Then at this point
	\begin{align*}
		0\le- r^{-\frac{7}{2}}\rho<0,
	\end{align*}
	a contradiction. Thus, for all $t\in[0,T]$ and $\rho > 0$, the following inequality holds
	\begin{align*}
		\mathcal{Q}(\theta,t)\le\max\limits_{\theta\in[0,2\pi]}\mathcal{Q}(\theta,0).
	\end{align*}
	This completes the proof.
\end{proof}
Since $\eqref{reduce}$ is a divergent quasilinear parabolic equation, classical theory (e.g., \cite{Lie96}) ensures higher regularity estimates follow from the uniform gradient estimate in Lemma \ref{C1}. Therefore, the solution $ r(\cdot,t)$ exists for all $t \in [0, +\infty)$. By Lemma \ref{equi}, we obtain
\begin{prop}\label{long}
	The flow \eqref{length pre flow}, starting from any smooth, closed, space-like curve on the light-cone $LC^*$,
	exists for all times and admits uniform $C^k$ estimates for all $k \in \mathbb{N}$.
\end{prop}

\subsection{Convergence}
To study the asymptotic behavior as $t\rightarrow+\infty$, consider the energy functional
\begin{align*}
	\mathcal{E}(t):=\oint r^2(\theta,t)\d\theta.
\end{align*}
\begin{lem}\label{E}
	Suppose $ r(\theta,t)$ is a solution of the initial value problem \eqref{reduce}. Then the energy $\mathcal{E}(t)$ satisfies
	\begin{align*}
		\frac{\d\mathcal{E}}{\d t}=-2\oint r^{-\frac{3}{2}} r_\theta^2\d\theta,
	\end{align*}
	which implies $\mathcal{E}(t)$ is non-increasing along the ﬂow.
\end{lem}
\begin{proof}
	Differentiating $\mathcal{E}(t)$ and integrating by parts
	\begin{align*}
		\frac{\d\mathcal{E}}{\d t}&=2\oint r\left( r^{-\frac{3}{2}} r_{\theta\theta}-\frac{3}{2} r^{-\frac{5}{2}} r_\theta^2\right)\d\theta\\
		&=2\oint r\left( r^{-\frac{3}{2}} r_\theta\right)_\theta\d\theta\\
		&=-2\oint r^{-\frac{3}{2}} r_\theta^2\d\theta.
	\end{align*}	
\end{proof}

\begin{lem}\label{derivate bound}
	Under the flow \eqref{length pre flow}, the following estimate holds
	\begin{align*}
		\left\lvert\frac{\d}{\d t}\oint r^{-\frac{3}{2}} r_\theta^2\d\theta\right\rvert\le C,
	\end{align*}
	where the constant $C$ depends only on the initial data $ r_0$. 
\end{lem}
\begin{proof}
	Using the evolution equations \eqref{0order} and \eqref{1order}, we compute
	\begin{align*}
		\frac{\d}{\d t}\oint r^{-\frac{3}{2}} r_\theta^2\d\theta&=\oint\left(-\frac{3}{2}r^{-\frac{5}{2}}r_\theta^2r_t+2r^{-\frac{3}{2}}r_\theta r_{\theta t}\right)\d\theta\\
		&=\oint\left(2r^{-3}r_\theta r_{\theta\theta\theta}-\frac{21}{2}r^{-4}r_\theta^2r_{\theta\theta}+\frac{39}{4}r^{-5}r_\theta^4\right)\d\theta.
	\end{align*}
	By Proposition \ref{long}, each term on the right-hand side is uniformly bounded in terms of the initial data, hence there exists a constant $C$ depending only on $r_0$ such that
	\begin{align*}
		\left\lvert\frac{\d}{\d t}\oint r^{-\frac{3}{2}} r_\theta^2\d\theta\right\rvert\le C.
	\end{align*}
\end{proof}

\begin{lem}\label{L2}
	Along the flow \eqref{length pre flow}, it holds that
	\begin{align*}
		\lim\limits_{t\rightarrow+\infty}\oint r_\theta^2\d\theta=0.
	\end{align*}
\end{lem}
\begin{proof}
	From Lemma \ref{E}, we have
	\begin{align}\label{conve}
		\int_{0}^{+\infty}\oint r^{-\frac{3}{2}} r_\theta^2\d\theta\d t<\infty.
	\end{align} 
	Combined with the uniform bound from Lemma \ref{derivate bound}, it follows that
	\begin{align}\label{int conv}
		\lim\limits_{t\rightarrow+\infty}\oint r^{-\frac{3}{2}} r_\theta^2\d\theta=0.
	\end{align}
	Suppose, for the sake of contradiction, that the above limit does not hold. Then there exists a constant $C_0 > 0$ and a sequence  $\{t_i\}_{i=1}^\infty$ tending to infinity, such that $\oint r^{-\frac{3}{2}} r_\theta^2\d\theta\ge C_0$ for all $i$. Since $\frac{\d}{\d t}\oint r^{-\frac{3}{2}} r_\theta^2\d\theta$ is uniformly bounded, we can find a constant $\epsilon_0$ such that for all
	$t\in(t_i,t_i+\epsilon_0)$, 
	\begin{align*}
		\oint r^{-\frac{3}{2}} r_\theta^2\d\theta\ge \frac{C_0}{2}.
	\end{align*}
	 Integrating over these intervals gives
	\begin{align}\label{cont}
		\int_{t_i}^{t_i+\epsilon_0}\oint r^{-\frac{3}{2}} r_\theta^2\d\theta\d t\ge\frac{C_0\epsilon_0}{2}.
	\end{align}
	However, \eqref{conve} implies $\lim\limits_{i\rightarrow\infty}\int_{t_i}^{+\infty}\oint r^{-\frac{3}{2}} r_\theta^2\d\theta\d t=0$, which contradicts \eqref{cont}. Therefore, the limit \eqref{int conv} must hold. Finally, since $r$ is uniformly bounded along the flow by Lemma \ref{C0}, we conclude that
		\begin{align*}
		\lim\limits_{t\rightarrow+\infty}\oint r_\theta^2\d\theta=0.
	\end{align*}
\end{proof}

We are now ready to complete the proof of Theorem \ref{thm behavior}.
\begin{proof}[Proof of Theorem \ref{thm behavior}]
	By Proposition \ref{long} and the Arzel\`a–Ascoli theorem, any sequence $\{t_k\}$ has a subsequence $\{t_{k_i}\}$ such that $ r_\theta(\theta, t_{k_i})$ converges smoothly to a function $f(\theta)$ as $t_{k_i} \to +\infty$. In view of Lemma \ref{L2}, the limit function must be identically zero, i.e., $f(\theta) \equiv 0$. Since every subsequence converges to zero, it follows that $ r_\theta$ converges smoothly to zero as $t \to +\infty$. Applying the same argument again and invoking Proposition \ref{L}, we conclude that $ r$ converges smoothly to a constant value $\omega$, where $\omega=\frac{1}{2\pi}\oint r_0(\theta)\d\theta$.
\end{proof}

\noindent {\bf Acknowledgements}
This work is supported by the National Natural Science Foundation of China (No. 12571062).\\
\\
\noindent {\bf Data availability} No data availability statement is required, as no experimental data is involved.

%\noindent{\bf Declarations}
\noindent{\bf Conflict of interest} On behalf of all authors, the corresponding author states that there is no conflict of
interest.

\end{document}